\documentclass[11pt]{amsart}

\usepackage{hyperref,amsmath,amsthm,amsfonts,amscd,flafter,epsf,amssymb,
ifsym,wasysym,color}
\usepackage{epsfig}
\usepackage{amsfonts}
\usepackage{amssymb}
\usepackage{amsmath}
\usepackage{amsthm}
\usepackage{latexsym}
\usepackage{amscd}
\usepackage{epsf}

\usepackage{mathtools}
\usepackage{stmaryrd}
\usepackage{todonotes}

\input{diagram}

\newcommand{\R}{\mathbb{R}}

\newcommand{\Z}{\mathbb{Z}}

\newcommand\Field{\mathbb F}

\newcommand{\im}{\mathrm{im}}

\newcommand{\Ht}{\mathrm{H}}

\newcommand\x{\mathbf x}

\newcommand{\ti}{\widetilde}

\newcommand{\ra}{\rightarrow}

\newtheorem{thm}{Theorem}[section]

\newtheorem{cor}[thm]{Corollary}

\newtheorem{lem}[thm]{Lemma}

\newtheorem{defn}[thm]{Definition}

\newtheorem{remark}[thm]{Remark}

\def\endproof{\relax\ifmmode\expandafter\endproofmath\else
  \unskip\nobreak\hfil\penalty50\hskip.75em\hbox{}\nobreak\hfil\bull
  {\parfillskip=0pt \finalhyphendemerits=0 \bigbreak}\fi}
\def\endproofmath$${\eqno\bull$$\bigbreak}
\def\bull{\vbox{\hrule\hbox{\vrule\kern3pt\vbox{\kern6pt}\kern3pt\vrule}\hrule}}

\makeatletter
\providecommand\@dotsep{5}
\def\listtodoname{List of Todos}
\def\listoftodos{\@starttoc{tdo}\listtodoname}
\makeatother

\newcommand{\fmap}{\mathfrak{f}}

\newcommand{\wu}{\mathfrak{u}}
\newcommand{\ord}{\mathrm{ord}}

\newcommand{\Cross}{\mathfrak{C}}
\newcommand{\CBn}{\mathcal{C}_{BN}}
\newcommand{\HBn}{\mathcal{H}_{BN}}
\newcommand{\dBn}{{\delta}_{BN}}
\newcommand{\grh}{\mathrm{gr}_h}
\newcommand{\grq}{\mathrm{gr}_q}

\newcommand{\alg}{\mathbf A}

\newcommand{\hsf}{\mathsf{h}}

\begin{document}

\title{The Bar-Natan homology and unknotting number}%
\author{Akram Alishahi}
\thanks{The author was supported by NSF Grant DMS-1505798.}
\address{Department of Mathematics, Columbia University, New York, NY 10027}
\email{alishahi@math.columbia.edu}

\date{\today}

\begin{abstract}
We show that the order of torsion homology classes in Bar-Natan deformation of Khovanov homology is a lower bound for the unknotting number. We give examples of knots that this is a better lower bound than $|s(K)/2|$, where $s(K)$ is the Rasmussen $s$ invariant defined by the Bar-Natan spectral sequence. \end{abstract}
\maketitle
\tableofcontents

\section{introduction}
In \cite{Kh-Intro}, Khovanov introduced a knot (and link) invariant which categorifies the Jones polynomial, now known as \emph{Khovanov homology}. This invariant is constructed by applying a specific TQFT to the cube of resolutions corresponding to a projection of the knot. Using a different TQFT, Bar-Natan defined a deformation of Khovanov homology in \cite{BarN}, which we will work with in this paper. The goal is to describe a lower bound for the unknotting number in terms of the \emph{h-torsion} in the Bar-Natan chain complex.

A homology class $\alpha\in\HBn(K)$ is called \emph{torsion} if $h^n.\alpha=0$ for a positive integar $n$. The smallest $n$ with this property is called \emph{order} of $\alpha$, denoted by $\ord(\alpha)$. Let $T_{BN}(K)$ denotes the set of torsion classes in $\HBn(K)$.

\begin{defn}\label{def:inv}
For an oriented knot $K$ in $\R^3$, we define 
\[\wu(K):=\max_{\alpha\in T_{BN}(K)}\ord(\alpha).\] 
\end{defn}

\begin{thm}\label{thm:lbunknotting}
For any oriented knot $K$, $\wu(K)$ is a lower bound for the unknotting number of $K$.
\end{thm}
Let $K_+$ and $K_-$ be knot diagrams that differ in a single crossing $c$, which is a positive crossing in $K_+$ and a negative crossing in $K_-$. We prove Theorem \ref{thm:lbunknotting} in by introducing chain maps 
\begin{equation}\label{eq:chmap}
f_c^{+}:\CBn(K_+)\to\CBn(K_-)\ \ \ \ \text{and}\ \ \ \ f_{c}^-:\CBn(K_-)\to\CBn(K_+)
\end{equation}
such that the induced maps by $f_c^-\circ f_c^+$ and $f_c^+\circ f_c^-$ on $\HBn(K_+)$ and $\HBn(K_-)$, respectively, are equal to multiplication by $h$. In \cite{AD-KM}, Dowlin and the author introduce similar chain maps for Lee homology and prove Knight Move Conjecture \cite{Kh-Intro,BarN2} for knots with unknotting number smaller than $3$.

Despite the algebraic definition of the chain maps (\ref{eq:chmap}), we show that they can be described in terms of cobordism maps associated to specific cobordisms from $K_+$ to $K_-\#H$, and $K_-$ to $K_+\#mH$, where $H$ is the right-handed Hopf link and $mH$ is its mirror. In \cite{AE-2}, Eftekhary and the author use corresponding cobordism maps for knot Floer homology to deduce a lower bound for the unknotting number, in terms of the order of torsion classes in variants of knot Floer homology.

This paper is organized as follows. Section \ref{sec:background} reviews Bar-Natan chain complex and collects some results we will need later. Section \ref{unknotting-bound} proves Theorem~\ref{thm:lbunknotting}. Section~\ref{sec:cobdescrip} gives a geometric description, using cobordism maps, for the chain maps, defined algebraically,  in the process of proving Theorem~\ref{thm:lbunknotting} in Section \ref{unknotting-bound}. Finally, Section~\ref{sec:Exam} gives examples of knots for which our invariant (Definition~\ref{def:inv}) is a better lower bound comparing to the $s$-invariant i.e. $\wu(K)>|s(K)|/2$.

{\it Acknowledgements.} The author would like to thank Robert Lipshitz for his comments, suggestions and helpful conversation; thanks also to Nathan Dowlin for helpful input.

\section{Background}\label{sec:background}
In this section we review the Bar-Natan chain complex, describe its module structure and discuss some of its basic properties. 

\subsection{Bar-Natan's deformation of Khovanov homology} Let $K$ be an oriented knot or link diagram in $\R^2$ with $n$ crossings. Denote the set of crossings in $K$ by $\Cross=\{c_1,...,c_n\}$. Each crossing can be resolved in two different ways, the $0$-resolution and the $1$-resolutions, see Figure \ref{fig:Resol}.

\begin{figure}[ht]
\centering
\def\svgwidth{8cm}
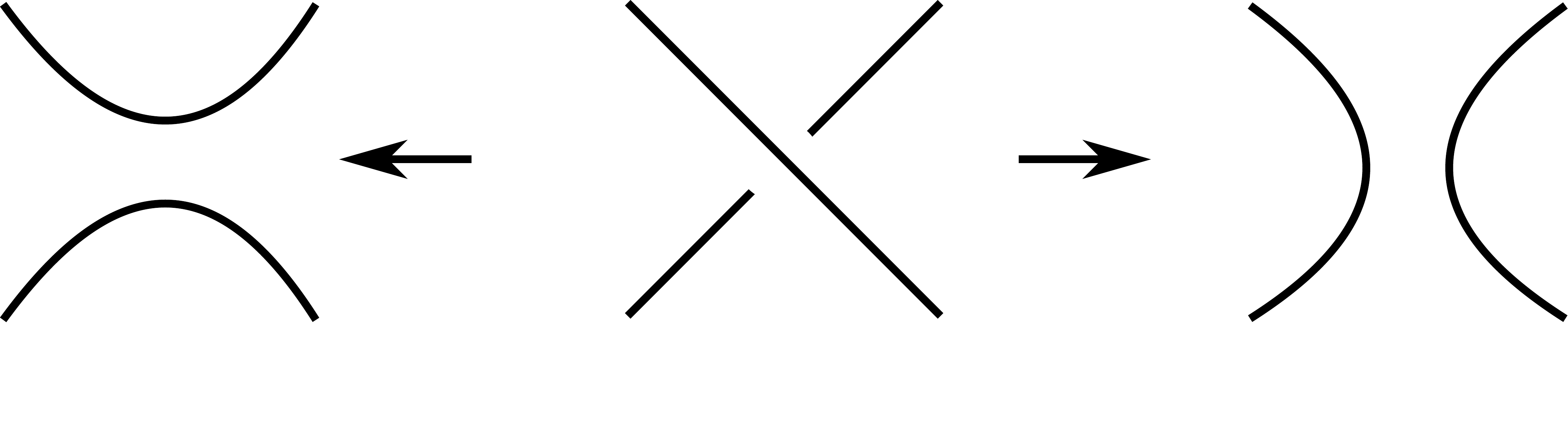
\caption{}\label{fig:Resol}
\end{figure}

For any vertex $v$ of $\{0,1\}^n$, let $K_v$ denote the complete resolution obtained by replacing the crossing $c_i$ by its $v_i$-resolution. Let $k_v$ denote the number of connected components of $K_v$.

There is a partial order on $\{0,1\}^n$ by setting $u\le v$ if $u_i\le v_i$ for all $1\le i\le n$. Denote $u\lessdot v$ if $u<v$ and $|v|-|u|=1$, where $|v|$ denotes $\sum_i v_i$. Corresponding to each edge of the cube, i.e. a pair $u\lessdot v$, there is an embedded cobordism in $\mathbb{R}^2\times [0,1]$ from $K_u$ to $K_v$, constructed by attaching an embedded one handle near the crossing $c_i$ where $u_i<v_i$.  If $k_u>k_v$, the cobordism \emph{merges} two circles, otherwise \emph{splits} two circles.

 Set $\Field=\Z/2\Z$. Let $\alg$ denotes the $2$-dimensional Frobenius algebra over $\Field[h]$ with basis $\{x_+,x_-\}$ and multiplication and  comultiplication defined as:
  \begin{minipage}[t]{0.4\textwidth}
  \begin{displaymath}
  \begin{split}
 &x_+\otimes x_+\xmapsto{m}  x_+\\
 &x_-\otimes x_+\xmapsto{m}  x_-\\
 &x_+\otimes x_-\xmapsto{m}  x_-\\
 &x_-\otimes x_-\xmapsto{m}  hx_-
\end{split}
  \end{displaymath}
 \end{minipage}
  \begin{minipage}[t]{0.6\textwidth}
  \begin{displaymath}
  \begin{split}
  &\\
&x_+\xmapsto{\Delta} x_+\otimes x_-+x_-\otimes x_++hx_+\otimes x_+\\
&x_-\xmapsto{\Delta} x_-\otimes x_-.
\end{split}
  \end{displaymath}
 \end{minipage}\\
 
 The Bar-Natan chain complex is obtained by applying the $(1+1)$-dimensional TQFT corresponding to $\alg$ to the above cube of cobordisms for $K$. More precisely, corresponding to a vertex $v\in\{0,1\}^n$, a Khovanov generator is a labeling of the circles in $K_v$ by $x_+$ or $x_-$. The module $\CBn(K_v)$ is defined as the free $\Field[h]$-module generated by the Khovanov generators corresponding to $v$ and 
 $$\CBn(K):=\bigoplus_{v\in \{0,1\}^n}\CBn(K_v).$$
 The differential $\dBn$ decomposes along the edges; for any $u\lessdot v$ the component
 $$\dBn^{u,v}:\CBn(K_u)\ra\CBn(K_v)$$ 
 is defined by the above multiplication if $K_v$ is obtained from $K_u$ by merge, otherwise, it is defined by comultiplication.  The Bar-Natan chain complex, $(\CBn(K),\dBn)$, was studied by Bar-Natan in~\cite{BarN}; the homology is denoted by $\HBn(K)$. For simplicity, we denote the differential by $\delta$.
 
The chain complex is bigraded; homological grading $\grh$ and an internal grading $\grq$ called \emph{quantum} grading. The homological grading for each summand $\CBn(K_v)$ of $\CBn(K)$ is given by $|v|-n_-$, where $n_\bullet$ denotes the number of $\bullet$-crossings in $K$ for $\bullet\in\{+,-\}$. The quantum grading for each Khovanov generator $x$ at a vertex $v$ is given by $$\grq(x)=n_+-2n_-+|v|+k_v^+-k_v^-$$
 where $k_v^{\bullet}$ denote the number of circles labelled by $x_{\bullet}$ in $K_v$, for $\bullet=+,-$. Furthermore, the formal variable $h$ has homological grading $0$ and quantum grading $-2$.\\
 

 \subsection{Module structure on Bar-Natan homology and basepoint action} Let $K$ be a knot diagram and $p$ be a point on $K$ away from the crossings. The choice of $p$, induces a module structure on the Khovanov homology of $K$, described in~\cite{Kh-patterns}. Let us recall this structure for Bar-Natan homology. Choose a small unknot $U$ near $p$ and disjoint from $K$ such that merging the unknot with $K$ gives a knot or link diagram isotopic to $K$. Then, attaching the corresponding embedded one handle to $K\sqcup U$ gives an embedded cobordism in $\R^3\times I$ from $K\sqcup U$ to $K$ and its associated cobordism map, denoted by $m_p$, $$m_p:\CBn(K\sqcup U)=\CBn(K)\otimes _{\Field[h]}\alg\ra \CBn(K)$$
 is given by the multiplication map $m$ of $\alg$. More precisely, for a Khovanov generator $x\in\CBn(K_v)$, $m_p(x\otimes x_{\bullet})$ is the Khovanov generator obtained from $x$ by multiplying the label of the circle containing $p$ with $x_\bullet$.
 
 Similarly, let 
 $$\Delta_p:\CBn(K)\to\CBn(K)\otimes_{\Field[h]}\alg$$
 denote the cobordism map associated to the inverse cobordism from $K$ to $K\sqcup U$.
 
 If $K$ is related to another knot diagram $K'$ by a Reidemeister move away from the basepoint $p$, it is straightforward that the chain homotopy equivalence between $\CBn(K)$ and $\CBn(K')$, defined in~\cite{BarN}, commutes with $m_p$. On the other hand, any Reidemeister move which crosses $p$ is equivalent to a sequence of Reidemeister moves away from $p$. So it induces an $\alg$-module structure on the Bar-Natan homology of the underlying knot $K$. 
 
For a point $p$ on $K$, let
 $$\x_p:\CBn(K)\ra\CBn(K)$$
be the chain map $\x_p(a)=m_p(a\otimes x_-)$, defined as in~\cite{HN-Khovanov}. Therefore, for a Khovanov generator $x\in \CBn(K_v)$ if the circle containing $p$ is labeled by $x_+$ then $\x_p(x)$ is the Khovanov generator obtained from $x$ by changing the label of this circle to $x_-$, otherwise $\x_p(x)=hx$. Thus, $\x_p\circ\x_p=h\x_p$ and $\x_p$ reduces the quantum grading by $2$. 
  
 In contrast to Khovanov homology,  the chain homotopy type of the module multiplication map $\x_p$ is not indepent of the marked point $p$. In fact, ~\cite[Lemma 2.3]{HN-Khovanov} may be generalized to describe the difference of $\x_p$ and $\x_q$ when $p$ and $q$ lie on opposite sides of a crossing as follows. 
 
\begin{lem}\label{changep}
Let $p,q\in K$ be points away from the crossings, that lie on the opposite sides of a single crossing as in Figure \ref{fig:cross}. Then, $\x_p+\x_q$ is homotopy equivalent to multiplication by $h$. 
 \end{lem}
 
 \begin{figure}[ht]
\centering
\def\svgwidth{2cm}
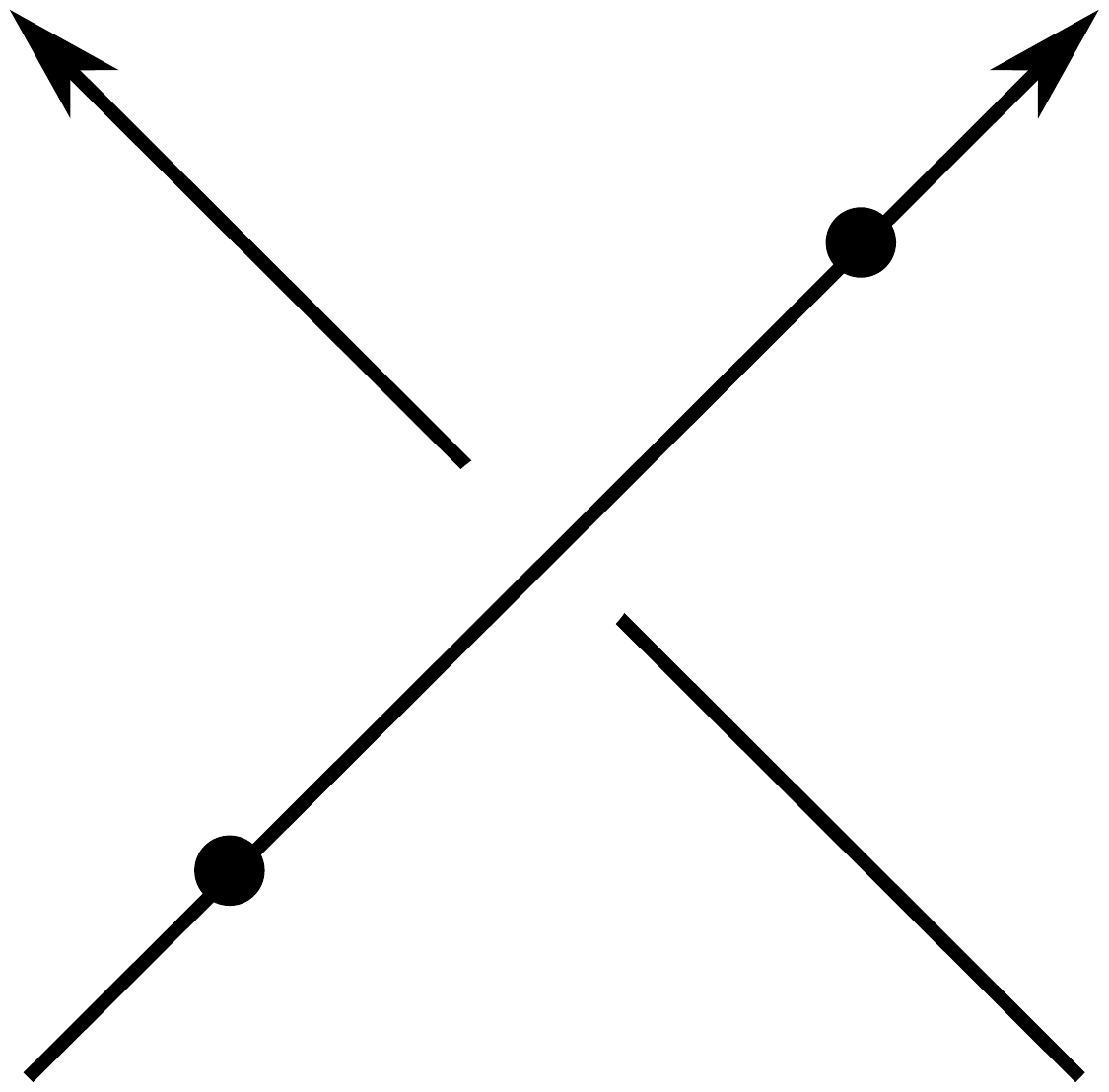
\caption{}\label{fig:cross}
\end{figure}   
 
 First we need to prove another lemma. Assume $K$ and $L$ be oriented link diagrams such that $L$ is obtained from $K$ by an oriented \emph{saddle} move as in Figure~\ref{fig:SaddleM}. This saddle move represents an oriented, embedded saddle cobordism in $\R^3\times [0,1]$ from $K$ to $L$. Let 
 $$\fmap:\CBn(K)\ra\CBn(L)\ \ \ \ \text{and}\ \ \ \ \bar{\fmap}:\CBn(L)\ra\CBn(K)$$ denote the chain maps on the Bar-Natan chain complex associated with this cobordism and its inverse, respectively. 
 
 \begin{figure}[ht]
\centering
\def\svgwidth{5.5cm}
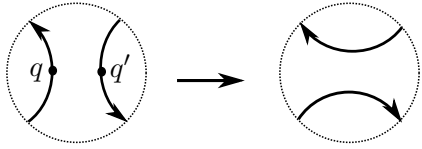
\caption{Saddle move: $q$ and $q'$ are the attaching points of the corresponding one-handle.}\label{fig:SaddleM}
\end{figure}  


\begin{lem}\label{Comp-Cob-vs-Inv} 
With the above notation fixed, for any $a\in\CBn(K)$ we have 
\begin{equation}\label{U1-handle}
\bar{\fmap}\circ\fmap(a)=ha+\x_q(a)+\x_{q'}(a).
\end{equation}
\end{lem}
\begin{proof} Given a vertex $v\in \{0,1\}^n$, if $q$ and $q'$ lie on the same connected component of the complete resolution $K_v$, then $\x_{q}|_{\CBn(K_v)}=\x_{q'}|_{\CBn(K_v)}$. Thus, for any $a\in\CBn(K_v)$ the Equality (\ref{U1-handle}) follows from
$$m\Delta(x_+)=hx_+,\ \ \ \ \text{and}\ \ \ \ m\Delta(x_-)=hx_-.$$ 

Otherwise, if $q$ and $q'$ belong to distinct connected components of $K_v$, for any Khovanov generator $x\in\CBn(K_v)$ the statement follows from one of the following relations, depending on the labels of the circles containing $q$ and $q'$.
\begin{displaymath}
\begin{split}
&\Delta m(x_+\otimes x_+)=x_+\otimes x_-+x_-\otimes x_++hx_+\otimes x_+\\
 &\Delta m(x_-\otimes x_+)=x_-\otimes x_-=x_-\otimes x_-+m(x_-\otimes x_-)\otimes x_++hx_-\otimes x_+\\
&\Delta m(x_+\otimes x_-)=x_-\otimes x_-=x_-\otimes x_-+x_+\otimes m(x_-\otimes x_-)+hx_+\otimes x_-\\
&\Delta m(x_-\otimes x_-)=hx_-\otimes x_-=hx_-\otimes x_-+m(x_-\otimes x_-)\otimes x_-+x_-\otimes m(x_-\otimes x_-)
\end{split}
\end{displaymath}
\end{proof}

\begin{proof}(Lemma~\ref{changep})
  The proof is similar to the proof of  ~\cite[Lemma 2.3]{HN-Khovanov}. Denote the crossing between $p$ and $q$ by $c$. Let $K_\bullet$ be the diagram obtained from $K$ after applying the $\bullet$-resolution at $c$. One may orient $K_0$ and $K_1$ such that $K_1$ is obtained from $K_0$ by an oriented saddle move, and up to appropriate grading shifts, $\CBn(K)$ is given by the mapping cone 
$$\fmap:\CBn(K_0)\ra\CBn(K_1)$$
where $\fmap$ is the corresponding cobordism map. Let $\bar{\fmap}$ denote the cobordism map associated with the inverse cobordism from $K_1$ to $K_0$. Under this decomposition we define 
$$H(a_0,a_1):=(\bar{\fmap}(a_1),0).$$
 Then, 
\begin{displaymath}
\begin{split}
\delta H(a_0,a_1)+H\delta (a_0,a_1)&=\delta (\bar{\fmap}(a_1),0)+H(\delta a_0,\fmap(a_0)+\delta a_1)\\
&=(\delta\bar{\fmap}(a_1),\fmap\bar{\fmap}(a_1))+(\bar{\fmap}(\fmap(a_0)+\delta a_1),0)\\
&=(\bar{\fmap}\fmap(a_0),\fmap\bar{\fmap}(a_1)),
\end{split}
\end{displaymath}
and thus it follows from lemma \ref{Comp-Cob-vs-Inv} that $H$ is a chain homotopy between $\x_p+\x_q$ and multiplication by $h$.
 \end{proof} 
 
 \begin{cor} \label{1-handle}
 Assume $K$ and $L$ are oriented link diagrams so that $L$ is obtained from $K$ by an oriented saddle move. If the attaching points of the saddle lie on the same connected component of $K$, then $\bar{\fmap}\circ\fmap$ is chain homotopic to multiplication by $h$. As before, $\fmap$ and $\bar{\fmap}$ denote the chain maps associated with the corresponding saddle cobordism and its inverse, respectively. 
 \end{cor}

 \begin{proof}
Let $p$ and $q$ be the attaching points of the saddle. Consider an arc $\alpha\subset K$ connecting $p$ to $q$. Moving the point $p$ along $\alpha$, it would cross an even number of crossings till it gets to $q$; thus Lemma \ref{changep} implies that $\x_p$ is homotopy equivalent to $\x_q$. Then, by Lemma \ref{Comp-Cob-vs-Inv} we have $\bar{\fmap}\fmap$ is homotopy equivalent to multiplication by $h$.

\end{proof}

\subsection{Connected sum with a Hopf link}\label{sec:sumHopf}
A connected sum formula for Khovanov homology has been studied in \cite{Kh-Intro}.  In this section, we recall a special case of this formula, taking a connected sum with Hopf link, for Bar-Natan homology. 
 
 
Let $H$ be the right-handed Hopf link. Choose an arbitrary point $p\in K$. We obtain an oriented diagram $L$ for the link $K\# H$ by changing $K$ locally in a neighborhood of $p$, as in Figure \ref{fig:Hopf}.

\begin{figure}[ht]
\centering
\def\svgwidth{5.5cm}
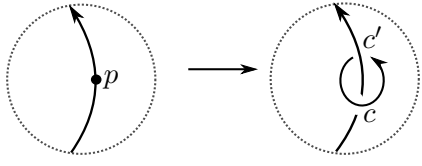
\caption{Connected sum with right-handed Hopf link at $p$}\label{fig:Hopf}
\end{figure}

Denote the new crossings by $c$ and $c'$. For $i,j\in\{0,1\}$, let $L_{ij}$ denote the diagram obtained from $L$ by applying $i$- and $j$-resolutions at the crossings $c$ and $c'$, respectively. We orient these diagrams such that their orientations coincide with the orientation of $K$ outside the above neighborhood of $p$. The oriented diagrams $L_{01}$ and $L_{10}$ are isotopic to $K$, while $L_{00}$ and $L_{11}$ are isotopic to a disjoint union of $K$ with an unknot near the point $p$. 
The chain complex $\CBn(L)$ is given by the mapping cone: 
\begin{displaymath}
\begin{diagram}
\CBn(L_{10})=\CBn(K)&&\rTo{\Delta_p}&&\CBn(L_{11})=\CBn(K)\otimes_{\Field[h]}\alg\\
\uTo{m_p}&&&&\uTo{\Delta_p}\\
\CBn(L_{00})=\CBn(K)\otimes_{\Field[h]}\alg&&\rTo{m_p}&&\CBn(L_{01})=\CBn(K)
\end{diagram}
\end{displaymath}

From now on, we use this decomposition to show any  $a\in\CBn(L)$ as $a=(a_{00},a_{01},a_{10},a_{11})$ where $a_{ij}\in\CBn(L_{ij})$ and so $a_{00},a_{11}\in\CBn(K)\otimes \alg$ while $a_{10},a_{01}\in\CBn(K)$.

 We define chain maps $i:\CBn(K)\to\CBn(L)$ and $p:\CBn(L)\to\CBn(K)$ as 
\begin{equation}\label{def:1stmaps}
i(a)=(0,0,0,a\otimes x_+)\ \ \ \ \text{and}\ \ \ \ p(a_{00},a_{01},a_{10},a_{11})=a_{00}^-
\end{equation}
where $a_{00}=a_{00}^+\otimes x_++a_{00}^-\otimes x_-$. It is straightforward that both $i$ and $p$ are chain maps. 

 \begin{lem}\label{lem:exactHopf}
The sequence 
\[
0\rightarrow \HBn(K)\xrightarrow{i_\star}\HBn(L)\xrightarrow{p_\star}\HBn(K)\rightarrow 0
\]
is a split exact sequence.
\end{lem}
 
\begin{proof} 
First, we prove that $\im(i_\star)=\ker(p_\star)$. Consider a homology class $\alpha$ in $\ker(p_\star)$. Any such class can be represented by a cycle $a$ such that $a_{00}^-=0$. Since, $\delta(a)=0$, $m_p(a_{00})+\delta(a_{01})=0$ and thus
\[a+\delta(a_{01}\otimes x_+,0,0,0)=(0,0,a_{01}+a_{10},a_{11}).\]
Again, follows from $\delta a=0$ that 
\[\delta a_{11}=\Delta_p(a_{10}+a_{01})=(a_{10}+a_{01})\otimes x_-+(\x_p(a_{10}+a_{01})+h(a_{10}+a_{01}))\otimes x_+.\] 
Therefore, $\delta a_{11}^-=a_{10}+a_{01}$ where $a_{11}=a_{11}^+\otimes x_++a_{11}^-\otimes x_-$ and so 
\[(0,0,a_{10}+a_{01},a_{11})+\delta(0,0,a_{11}^-,0)=(0,0,0,a_{11}+\Delta_p a_{11}^-)=i(a_{11}^++\x_p(a_{11}^-)+ha_{11}^-).\]

Then, let $r:\CBn(L)\to\CBn(K)$ and $s:\CBn(K)\to\CBn(L)$ be the chain maps defined as 
\[r(a_{00},a_{01},a_{10},a_{11})=a_{11}^++\x_p(a_{11}^-)+ha_{11}^-\ \ \ \ \text{and}\ \ \ \ s(a)=(\x_p(a)\otimes x_++a\otimes x_-,0,0,0)\]
where $a_{11}=a_{11}^+\otimes x_++a_{11}^-\otimes x_-$. It is straightforward that $r$ and $s$ are chain maps such that $r\circ i=\mathrm{id}$ and $p\circ s=\mathrm{id}$. So $i_\star$ and $p_\star$ are injective and surjective, respectively, and the sequence splits.

%
%
%
%
\end{proof}


The homomorphism $p_\star$ preserves homological grading and decreases the quantum grading by $1$, while $i_\star$ increases homological grading by $2$ and quantum grading by $5$. 

Similarly, we may define chain maps $i$ and $p$ for the connected sum of $K$ with the left-handed trefoil, $K\#mH$, so that the induced homomorphism on homology gives a split exact sequence. The only difference is that $p_\star$ increases the homological grading by $2$ and quantum grading by $5$, while $i_\star$ preserves the homological grading and decreases the quantum grading by $1$. 

Finally, we define the chain maps $\mathsf{i}:\CBn(K)\to\CBn(L)$ and $\mathsf{p}:\CBn(L)\to\CBn(K)$ as
\begin{equation}\label{def:Hopfmaps}
\mathsf{i}=i+s\ \ \ \ \text{and}\ \ \ \ \mathsf{p}=p+r
\end{equation}
where the chain maps $i$ and $p$ are defined in Equation~(\ref{def:1stmaps}) and $r$ and $s$ are defined in the proof of Lemma~\ref{lem:exactHopf}. Note that for $L=K\#mH$, one may define similar chain maps, abusing the notation, we denote these maps by $\mathsf{i}$ and $\mathsf{p}$, too.

\section{Lower bound for unknotting number}\label{unknotting-bound}
The goal of this section is to prove Theorem \ref{thm:lbunknotting}.

Let $C$ be a chain complex of $\Field[h]$-modules. Recall that a homology class $\alpha\in\Ht_\star(C)$ is called \emph{torsion} if $h^n\alpha=0$ for some positive $n$, and the smallest such $n$ is called the \emph{order} of $\alpha$, denoted by $\ord(\alpha)$. Let $T(C)$ be the set of torsion homology classes in $\Ht_\star(C)$ and define
\[\wu(C):=\max_{\alpha\in T(C)}\ord(\alpha).\]
\begin{lem}\label{lem:Algver}
  Given chain complexes $C$ and $C'$ of $\Field[h]$-modules, together with chain maps
  \[f:C\to C'\ \ \ \ \text{and}\ \ \ \ g:C'\to C\]
  so that both $f_\star \circ g_{\star}$ and $g_\star\circ f_\star$ are equal to multiplication by $h^n$ for some $n>0$, then
  \[|\wu(C)-\wu(C')|\le n.\]

\end{lem}

\begin{proof}
  For any homology class $\alpha\in T(C)$ we have $f_\star(\alpha)\in T(C')$, and it follows from $g_\star\circ f_\star(\alpha)=h^n\alpha$ that
\[\ord(h^n\alpha)\le\ord(f_\star(\alpha))\le \ord(\alpha).\]
Thus, $\ord(\alpha)\le\ord(f_\star(\alpha))+n$, and so $\wu(C)\le \wu(C')+n$. Similarly, $\wu(C')\le \wu(C)+n$ which proves the result.
  
\end{proof}

Suppose $K_+$ and $K_-$ be oriented knot diagrams so that $K_-$ is obtained from $K_+$ by changing one  positive crossing, denoted by $c$, to a negative crossing. For $i=0,1$, denote the $i$-resolution of $K_+$  at the crossing $c$ by $K_i$. We orient $K_0$ and $K_1$ such that they are related by an oriented saddle move. Let $\fmap$ and $\bar{\fmap}$ denote the cobordism maps corresponding to the saddle cobordism from $K_0$ to $K_1$ and its inverse from $K_1$ to $K_0$, respectively. The Bar-Natan chain complexes $\CBn(K_+)$ and $\CBn(K_-)$, upto grading shifts,  are given by the mapping cones of $\fmap$ and $\bar{\fmap}$ respectively.



Choose points $p$ and $q$ on the opposite sides of the crossing $c$, as in Figure \ref{fig:cross} and define

\begin{equation}\label{def:chainmaps}
\begin{array}{ll}
f_c^+:\CBn(K_+)\to\CBn(K_-)&f_c^-:\CBn(K_-)\to\CBn(K_+)\\
f_c^+(a_0,a_1)=\left((\x_p+\x_q)(a_1),a_0\right),&f_c^-(a_1,a_0)=\left((\x_p+\x_q)(a_0),a_1\right)
\end{array}
\end{equation}
where $a_i\in\CBn(K_i)$.
\begin{lem}
Both $f_c^+$ and $f_c^-$  are  chain maps.
\end{lem}

\begin{proof}
Let $\delta_\bullet$ denote the differential of $\CBn(K_\bullet)$ for $\bullet\in\{0,1,+,-\}$. It follows from $(\x_p+\x_q)\circ\fmap=0$ and $\bar{\fmap}\circ(\x_p+\x_q)=0$ that
  \[
    \begin{split}
      f_c^+\delta_+(a_0,a_1)&=\left((\x_p+\x_q)(\delta_1(a_1)+\fmap(a_0)),\delta_0(a_0)\right)\\
      &=\left((\x_p+\x_q)\delta_1(a_1),\delta_0(a_0)\right)= \delta_-f_c^+(a_0,a_1).
    \end{split}
  \]
  The proof for $f_c^-$ is similar.
\end{proof}
\begin{cor}\label{cor:singcross}
  With the above notation fixed, $|\wu(K_+)-\wu(K_-)|\le 1$.
  \end{cor}
 \begin{proof}
  By Lemma~\ref{changep} the induced maps on homology by both $f_c^+\circ f_c^-$ and $f_c^-\circ f_c^+$ are equal to multiplication by $h$. Thus the claim follows from Lemma~\ref{lem:Algver}.
    \end{proof}

    \begin{proof}(Theorem~\ref{thm:lbunknotting}) Consider an oriented diagram for $K$ such that we get an oriented diagram for the unknot after switching $N$ crossings $\{c_1,...,c_N\}$, where $N$ is the unknotting number of $K$. Abusing the notation we denote the diagram by $K$. For any $i=1,...,N$, let $K_i$ be the diagram obtained from $K$ after switching the crossings $c_1,...,c_i$. The diagrams $K_{i-1}$ and $K_i$ differ in a single crossing for each $i$, so  it follows from Corollary \ref{cor:singcross} that $|\wu(K_{i-1})-\wu(K_i)|\le 1$. Therefore, $ |\wu(K)-\wu(Unknot)|=\wu(K)\le N.$

  \end{proof}

\begin{remark}
Setting $h=1$, one may think of $(\CBn(K),\delta)$ as a filtered chain complex of $\Field$-modules, where the differential increases homological grading by $1$ and does not decrease quantum grading. This gives a spectral sequence from Khovanov homology with coefficients in $\Field$ of $K$ to $\Field\oplus \Field$, called Bar-Natan spectral sequence \cite{Tur-BN}. If this spectral sequence collapses in the n-th page, then $\wu(K)=n-1$.
\end{remark}

\section{A geometric interpretation of chain maps}\label{sec:cobdescrip}

Suppose $K$ and $K'$ are oriented pointed knots i.e. oriented knots with marked points on them, such that $K'$ is obtained from $K$ by a sequence of crossing changes. To any such sequence, Eftekhary and the author associate a decorated cobordism from $K$ to a connected sum of $K'$ with some right- or left-handed Hopf links. Then, by the corresponding cobordism maps for knot Floer homology~\cite[Section 8.2]{AE-2}, we define chain maps between the knot Floer chain complexes of $K$ and $K'$ satisfying the assumptions of Lemma~\ref{lem:Algver} ~\cite[Section 8.3]{AE-2}. As a result, one gets a lower bound for the unknotting number in term of the \emph{u-torsion} in knot Floer homology. 

Following the approach in~\cite{AE-2} one may use cobordism maps for Bar-Natan homology to define chain maps between $\CBn(K)$ and $\CBn(K')$ which satisfy the assumptions of Lemma~\ref{lem:Algver}. The goal of this section is to show that for any crossing change process, these chain maps are equal to the ones defined in Section \ref{unknotting-bound}.

As before, let $K_+$ be an oriented knot diagram with a specific positive crossing $c$, and $K_-$ be the oriented knot diagram obtained from $K_+$ by changing $c$ into a negative crossing. As in Figure~\ref{fig:RII}, after a Reidemeister II move near the crossing $c$ on $K_+$, followed by an oriented saddle move, one gets a diagram for $K_-\#H$. Here $H$ is the right-handed Hopf link. 

\begin{figure}[ht]
\centering
\def\svgwidth{13cm}
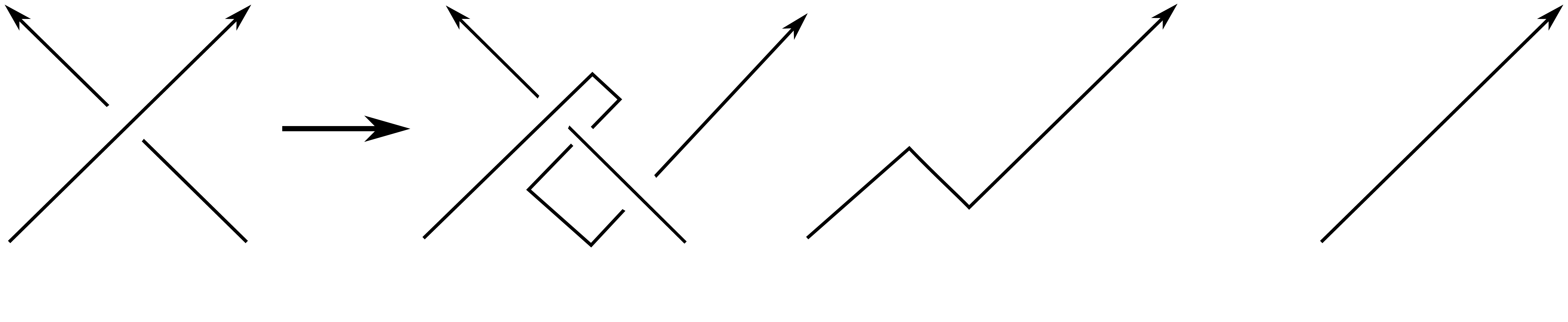
\caption{}\label{fig:RII}
\end{figure}  

As in Figure~\ref{fig:RII}, we denote the knot diagram obtained from $K_+$ by the specified Reidemeister II move by $\ti{K}_+$.   Further, let $\mathsf{h}:\CBn(K_+)\to\CBn(\ti{K}_+)$ and $\ti{\mathsf{h}}:\CBn(\ti{K}_+)\to\CBn(K_+)$ be the chain homotopy equivalences corresponding to this move as defined in~\cite[Section 4.3]{BarN}. For the reader's convenience, we recall the definition of $\mathsf{h}$ and $\ti{\mathsf{h}}$. For $i,j=0,1$, let $\CBn^{ij}(\ti{K}_+)$ denotes the direct sum of the summands of $\CBn(\ti{K}_+)$ corresponding to vertices $v$ of the cube so that $v(c_2)=i$, $v(c_3)=j$. Also, let $\mathsf{h}^{ij}=\pi_{ij}\circ\mathsf{h}$ and $\ti{\mathsf{h}}^{ij}=\ti{\mathsf{h}}\circ\imath_{ij}$ where $\pi_{ij}$ is the projection of $\CBn(\ti{K}_+)$ on $\CBn^{ij}(\ti{K}_+)$ and $\imath_{ij}$ is the inclusion of $\CBn^{ij}(\ti{K}_+)$ in $\CBn(\ti{K}_+)$. Then, $\mathsf{h}^{00}=\mathsf{h}^{11}=\ti{\mathsf{h}}^{00}=\ti{\mathsf{h}}^{11}=0$ and $\mathsf{h}^{10}=\ti{\mathsf{h}}^{10}=\mathrm{id}$. Further, $\mathsf{h}^{01}=g\otimes x_+$, where $g$ is the cobordism map corresponding to the saddle move as in Figure \ref{fig:RIIhom} i.e. $\mathsf{h}^{01}$ is a chain map for the cobordism which is union of a saddle and a cup.  Finally, $\ti{\mathsf{h}}^{01}$ is the cobordism map for the inverse of the saddle move in Figure \ref{fig:RIIhom} union a cap. 

\begin{figure}[ht]
\centering
\def\svgwidth{5.5cm}
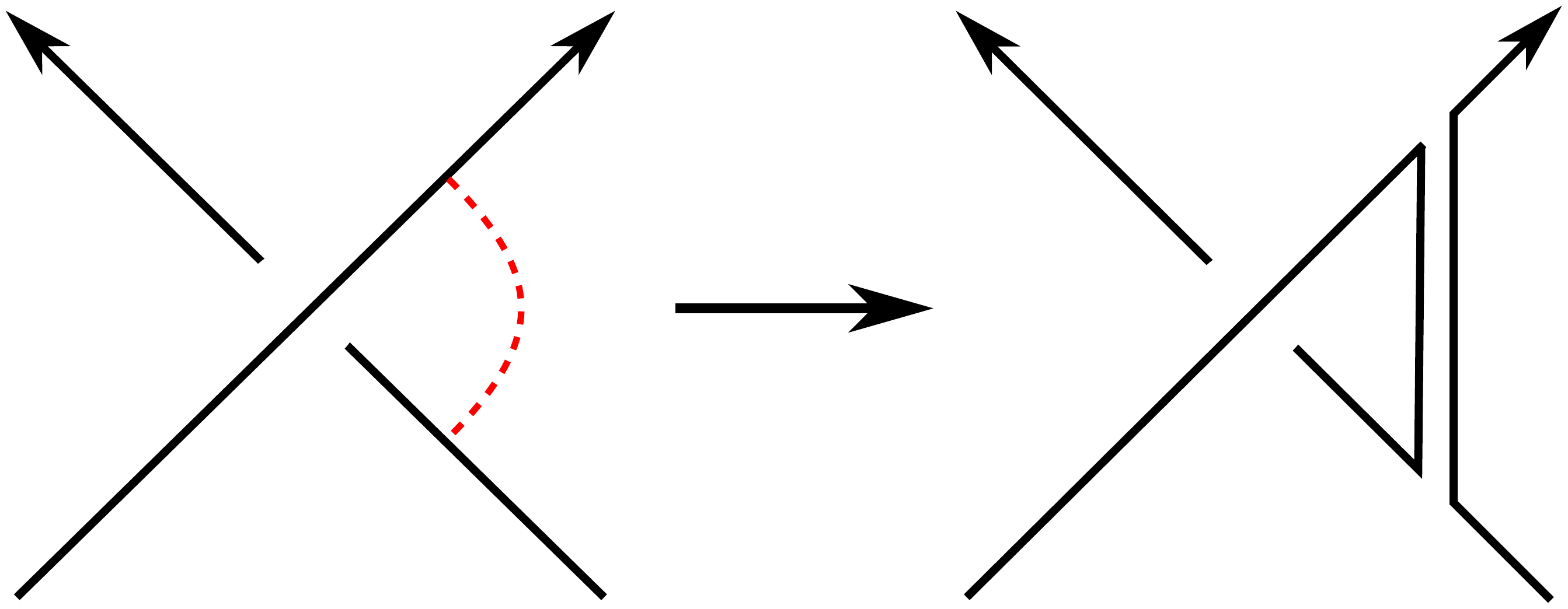
\caption{}\label{fig:RIIhom}
\end{figure}  

Let $\fmap:\CBn(\ti{K}_+)\to\CBn(K_-\#H)$ and $\bar{\fmap}:\CBn(K_-\#H)\to\CBn(\ti{K}_+)$ denote the cobordism maps for the saddle move in Figure \ref{fig:RII}. 

\begin{thm}
With the above notation fixed, $f_c^+=\mathsf{p}\circ \fmap\circ \mathsf{h}$ and $f_{c}^-=\ti{\mathsf{h}}\circ \bar{\fmap}\circ\mathsf{i}$, where $\mathsf{i}$ and $\mathsf{p}$ are the chain maps defined in (\ref{def:Hopfmaps}).
\end{thm}

\begin{proof}
We prove both equalities by looking at the cube of resolutions for the three crossings $c_1$, $c_2$ and $c_3$. 
For $i,j,k=0,1$, let $\CBn^{ijk}(\bullet)$ denotes the direct sum of the summands of $\CBn(\bullet)$ corresponding to vertices $v$ of the cube so that $v(c_1)=i$, $v(c_2)=j$ and $v(c_3)=k$ for $\bullet=\ti{K}_+,K_-\#H$. 
Similarly, $\CBn^{i}(K_+)$ and $\CBn^i(K_-)$ denote the summands of $\CBn(K_+)$ and $\CBn(K_-)$,  respectively, corresponding to the $i$-resolution at $c$. 

Assume $a=(a_0,a_1)$ be an element in $\CBn(K)$ so that $a_i\in\CBn^i(K_+)$. It follows from the definition of $\mathsf{h}$ that $\mathsf{h}^{ii}(a_j)=0$ for any $i,j=0,1$. Further, considering the definition of $\mathsf{p}$, it is enough to compute $\fmap(\mathsf{h}^{01}(a_0))$ and $\fmap(\mathsf{h}^{10}(a_1))$. As in Figure \ref{fig:0res}, $\mathsf{h}^{01}(a_0)=\Delta_p(a_0)\otimes x_+$ and so 
\begin{equation}\label{eq:formu1}
\fmap(\mathsf{h}^{01}(a_0))=\Delta_p(a_0)=a_0\otimes x_-+(\x_p(a_0)+ha_0)\otimes x_+
\end{equation}

\begin{figure}[ht]
\centering
\def\svgwidth{13cm}
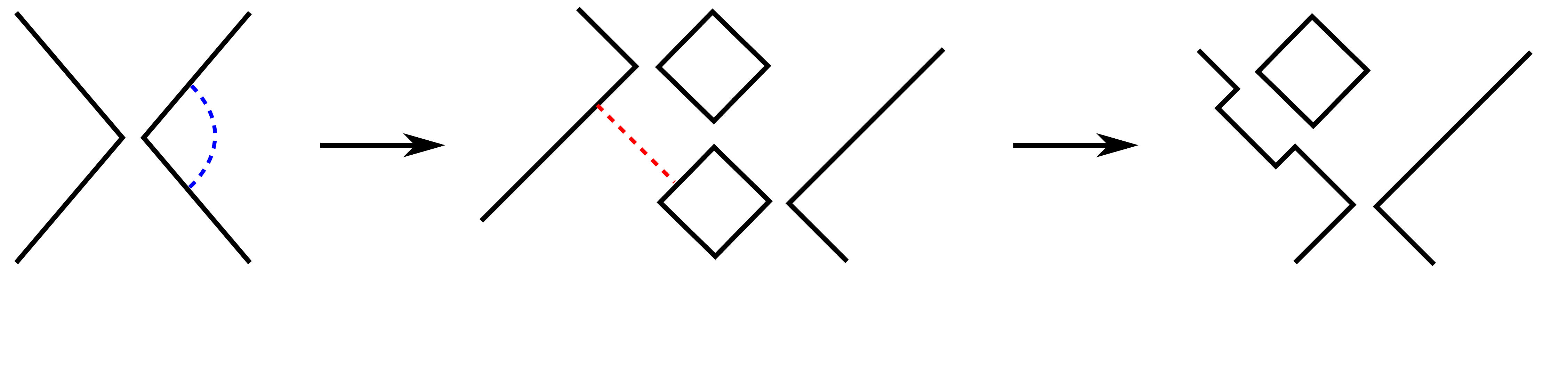
\caption{}\label{fig:0res}
\end{figure}

Similarly, Figure \ref{fig:1res} shows that $\mathsf{h}^{10}(a_1)=a_1$ and 
 \begin{equation}\label{eq:formu2}
\fmap(\mathsf{h}^{10}(a_1))=\Delta_q(a_1)=a_1\otimes x_-+(\x_q(a_1)+ha_1)\otimes x_+
\end{equation}
Thus, by equalities (\ref{eq:formu1}) and (\ref{eq:formu2}) we have $\mathsf{p}(\fmap\circ \mathsf{h}(a))=(\x_p(a_1)+\x_q(a_1),a_0)$.

\begin{figure}[ht]
\centering
\def\svgwidth{13cm}
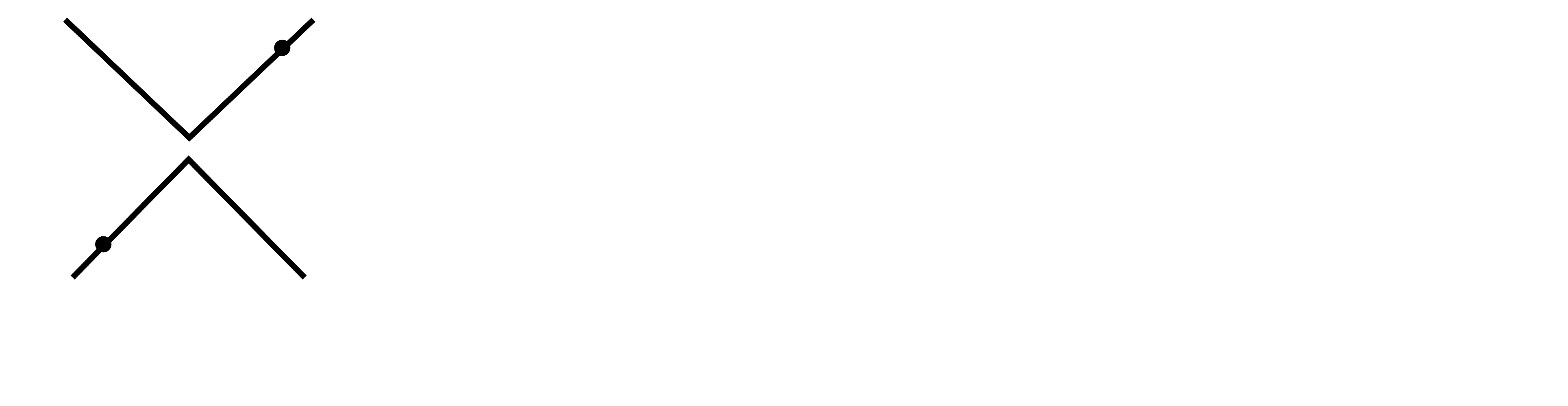
\caption{}\label{fig:1res}
\end{figure}

Let $a=(a_0,a_1)$ be an element in $\CBn(K_-)$ where $a_i\in\CBn^i(K_-)$ for $i=0,1$. Denote the components of $\mathsf{i}(a_\star)$ in $\CBn^{jk\star}(K_-\#H)$ by $\mathsf{i}^{jk}(a)$. It follows from the definitions of $\ti{\mathsf{h}}$ and $\mathsf{i}$ that to compute 
$\ti{\mathsf{h}}(\bar{\fmap}\circ \mathsf{i}(a))$, it is enough to compute:
\[\bar{\fmap}(\mathsf{i}^{11}(a_0))=\bar{\fmap}(a_0\otimes x_+)=a_0\]
and 
\[\bar{\fmap}(\mathsf{i}^{00}(a_1))=\bar{\fmap}(\x_q(a_1)\otimes x_++a_1\otimes x_-)=\Delta_q(\x_q(a_1)\otimes x_++a_1\otimes x_-).\]
It is not hard to see that $\ti{\mathsf{h}}(\bar{\fmap}\circ\mathsf{i}^{11}(a_0))=a_0$ and $\ti{\mathsf{h}}(\bar{\fmap}(\mathsf{i}^{00}(a_1)))=\x_p(a_1)+\x_q(a_1)$. This completes the proof. 
\end{proof}
\begin{remark}
Similarly, one may change $K_-$ by a Reidemeister II move near $c$ to get a diagram $\ti{K}_-$, so that $K_+\#mH$ is obtained from $\ti{K}_-$ by a saddle move. Let $\mathsf{h}:\CBn(K_-)\to\CBn(\ti{K}_-)$ and $\ti{\mathsf{h}}:\CBn(\ti{K}_-)\to\CBn(K_-)$ be the corresponding chain homotopy equivalences and $\fmap:\CBn(\ti{K}_-)\to \CBn(K_+\#mH)$ and $\bar{\fmap}:\CBn(K_+\#mH)\to\CBn(\ti{K}_-)$ be the cobordism maps for the saddle move. Then, by the same argument, one can show that $f_c^-=\mathsf{p}\circ \fmap\circ \mathsf{h}$ and $f_{c}^+=\ti{\mathsf{h}}\circ \bar{\fmap}\circ\mathsf{i}$. \end{remark}

\section{Examples}\label{sec:Exam}
The Rasmussen's s invariant gives a lower bound for the slice genus, $|s(K)|/2$, and thus the unknotting number \cite{Ras-sinv}. We used Cotton Seed's package, Knotkit \cite{S-knotkit}, to compute $\wu$ and s (defined using the Bar-Natan spectral sequence) for some knots with more than 12 crossings. We found some example that $\wu$ is a better lower bound comparing to $|s|/2$, for instance, $|s|/2$ for the knots $13n689, 13n1166, 13n2504$ and $13n2807$ is equal to $1$, while $\wu$ is equal to $2$.


\bibliographystyle{alpha}
\bibliography{HFbibliography}

\begin{thebibliography}{Kho03}

\bibitem[AD]{AD-KM}
Akram Alishahi and Nathan Dowlin.
\newblock The {L}ee spectral sequence, unknotting number, and the knight move
  conjecture.
\newblock preprint.

\bibitem[AE]{AE-2}
Akram Alishahi and Eaman Eftekhary.
\newblock Tangle {F}loer homology and cobordisms between tangles.
\newblock ArXiv:1610.07122.

\bibitem[BN02]{BarN2}
Dror Bar-Natan.
\newblock On {K}hovanov's categorification of the {J}ones polynomial.
\newblock {\em Alg. Geom. Top.}, 2:337--370, 2002.

\bibitem[BN05]{BarN}
Dror Bar-Natan.
\newblock Khovanov's homology for tangles and cobordisms.
\newblock {\em Geom. Topol.}, 9:1443--1499, 2005.

\bibitem[HN13]{HN-Khovanov}
Matthew Hedden and Yi~Ni.
\newblock Khovanov module and the detection of unlinks.
\newblock {\em Geom. Topol.}, 17(5):3027--3076, 2013.

\bibitem[Kho00]{Kh-Intro}
Mikhail Khovanov.
\newblock A categorification of the {J}ones polynomial.
\newblock {\em Duke Math. J.}, 101(3):359--426, 2000.

\bibitem[Kho03]{Kh-patterns}
Mikhail Khovanov.
\newblock Patterns in knot cohomology. {I}.
\newblock {\em Experiment. Math.}, 12(3):365--374, 2003.

\bibitem[Ras10]{Ras-sinv}
Jacob Rasmussen.
\newblock Khovanov homology and the slice genus.
\newblock {\em Invent. Math.}, 182(2):419--447, 2010.

\bibitem[See]{S-knotkit}
Cotton Seed.
\newblock Knotkit.
\newblock https://github.com/cseed/knotkit.

\bibitem[Tur06]{Tur-BN}
Paul~R. Turner.
\newblock Calculating {B}ar-{N}atan's characteristic two {K}hovanov homology.
\newblock {\em J. Knot Theory Ramifications}, 15(10):1335--1356, 2006.

\end{thebibliography}

\end{document}